\documentclass[12pt]{amsart}
\usepackage{mathrsfs}

\usepackage{amssymb,amsfonts,amsthm,amsmath}
\usepackage{epsfig}
\usepackage[utf8]{inputenc}
\usepackage{mathrsfs}
\usepackage{lscape}
\usepackage{amssymb,amsmath,graphicx,color,textcomp, amsthm,bbm,bbold, enumerate,booktabs}
\usepackage{chngcntr}
\usepackage{apptools}
\AtAppendix{\counterwithin{theorem}{section}}
\definecolor{Red}{cmyk}{0,1,1,0}

\definecolor{verde}{cmyk}{1,0,1,0}

\definecolor{loka}{cmyk}{.5,0,1,.5}
\definecolor{azul}{cmyk}{1,1,0,0}

\numberwithin{equation}{section}

\newcommand{\be}{\begin{equation}}
\newcommand{\ee}{\end{equation}}

\newtheorem{definition}{Definition}

\newtheorem{teorema}{Theorem}
\newtheorem{remark}{Remark}

\begin{document}
\title{Stability of $\psi$-Hilfer Impulsive Fractional Differential Equations}
\author{J. Vanterler da C. Sousa}
\address{ Department of Applied Mathematics, Institute of Mathematics,
 Statistics and Scientific Computation, University of Campinas --
UNICAMP, rua S\'ergio Buarque de Holanda 651,
13083--859, Campinas SP, Brazil\newline
e-mail: {\itshape \texttt{vanterlermatematico@hotmail.com, capelas@ime.unicamp.br }}}

\address{$^1$ Department of Mathematics, Shivaji University, Kolhapur 416 004, Maharashtra, India \newline
e-mail: {\itshape \texttt{kdkucche@gmail.com}}}
\author{Kishor D. Kucche$^1$}
\author{E. Capelas de Oliveira}

\begin{abstract} In this paper, we investigate the sufficient conditions for existence and uniqueness of solutions and $\delta$-Ulam-Hyers-Rassias stability of an impulsive fractional differential equation involving  $\psi$-Hilfer fractional derivative. Fixed point approach is used to obtain our main results.

\end{abstract}
\maketitle

\section{Introduction}

Impulsive differential equations  are used to describe the evolutionary
processes that abruptly change their state at a certain moment. This subject received great importance and remarkable attention from the researchers because of its rich theory  \cite{theory} and applicability in  various branches of science and technology. A natural framework for mathematical modeling of many physical phenomena appearing in the field of  mechanics, ecology, medicine, biology and electrical engineering can be provided  via impulsive differential equations.

Wang and Zhang \cite{WangZhang} investigated the existence and uniqueness of
solutions to differential equations with not instantaneous impulses in a $P\beta$-normed space of the form:
\begin{equation*}
\begin{cases} 
x'(t) =  f\left(t,x(t)\right),\, t\in(s_{i},t_{i+1}],\;
i=0,1,\cdots,m,\\
x(t)  =  g_{i}(t,x(t)), ~ t\in(t_{i},s_{i}],\; i=1,2,\cdots ,m.
\end{cases}
\end{equation*}
Wang et al. in \cite{wangJin}, considering the ordinary nonlinear differential equation with not instantaneous impulses obtained existence and uniqueness of solutions and introduced an interesting concept of stability viz. generalized $\beta$-Ulam-Hyers–Rassias. Zeng et al. \cite{princi} extended the above investigations  to  the class of impulsive integro-differential equations with not instantaneous impulses. 

With the expansion of the fractional calculus \cite{SAMKO,KSTJ,ZE1}, the impulsive fractional differential equations gained a much attention and began to be studied, mainly due to the variety of results, from the stability study, existence to uniqueness \cite{dicheng, stamova,liu,esto9,esto10,palm}. Impulsive 
differential equations in the space of Banach spaces have been dealt in \cite{esto3,esto5,esto11,exis7}. Discussions on concept of solutions, existence and uniqueness results pertaining to impulsive fractional Cauchy differential equations (IFDE) can also be found in the survey paper \cite{Wang01}.  

Wang and Zhang \cite{WangZh} recently, introduced a new class of nonlinear differential equations  with fractional integrable impulses and  established  existence and uniqueness results and introduced new concepts of Bielecki-Ulam's type stability.

As discussed above, in the last years there are many published works related to stability, especially on impulsive differential equations of arbitrary order. Thus, one of the mains purpose of this paper is to expand this range of works in the literature involving stability and thus contribute to the expansion of the area.

In this paper, we apply fixed point approach to study stability of the modified impulsive fractional differential equations 
\begin{equation}\label{eq1}
\left\{ 
\begin{array}{cll}
^{H}\mathbb{D}_{0+}^{\alpha ,\beta ;\psi }x\left( t\right) & = & f\left( t,x\left(t\right) \right) ,\text{ }t\in \left( s_{i},t_{i+1}\right] ,\text{ }i=0,1,...,m \\ 
x\left( t\right) & = & g_{i}\left( t,x\left( t_{i}^{+}\right) \right) ,\text{ }t\in \left( t_{i},s_{i}\right] \,,\text{ }i=1,2,...,m
\end{array}
\right.
\end{equation}
where $^{H}\mathbb{D}_{0+}^{\alpha ,\beta ;\psi }\left( \cdot \right) $ is the $\psi - $Hilfer fractional derivative with $0<\alpha \leq 1$, $0\leq \beta \leq 1$ and $0=t_{0}=s_{0}<t_{1}\leq s_{1}\leq t_{2}<\cdot \cdot \cdot <t_{m}\leq s_{m}<t_{m+1}=T$ are prefixed numbers, $f:\left[ 0,T\right] \times \mathbb{R} \rightarrow \mathbb{R}$ is continuous and $g_{i}$ $:\left[ t_{i},s_{i}\right] \times \mathbb{R}\rightarrow \mathbb{R}$ is continuous for all $i=1,2,...,m$ which is not instantaneous impulses.

The motivation for the elaboration of this paper is the contribution in the stability of fractional differential equations, in particular of the impulsive type. In this sense, as the main purpose of this paper, we investigated the $\delta$-Ulam-Hyers-Rassias stability of the impulsive fractional differential equation by employing the fixed point approach.

For the fundamental properties of $\psi$-Hilfer fractional derivative and the basic theory of fractional differential equation involving  $\psi$-Hilfer fractional derivative, we refer the readers to the papers of Sousa and Oliveira \cite{ZE1,ZE2}.

This paper is divided as follows: in Section 2, we present the concepts of weighted and piecewise weighted function spaces.  We also recall the definitions of $\psi$-Riemann-Liouville fractional integral, $\psi$-Hilfer fractional derivative and define the concept of generalized $\delta$-Ulam-Hyers-Rassias stability. In Section 3, we investigate through the Theorem \ref{teo2} the generalized $\delta$-Ulam-Hyers-Rassias stability of the fractional differential equation.


\section{Preliminaries}
\begin{definition} {\rm\cite{wangJin}}
Suppose $E$ is a vector space over $\mathbb{K}$. A function $\left\Vert \cdot\right\Vert _{\delta }:E\rightarrow \left[ 0,\infty \right) $ $ \left(0< \delta \leq 1\right) $ is called a $\delta-$norm if and only if it satisfies :
\begin{enumerate}
\item $\left\Vert x\right\Vert _{\delta }=0,$ if and only if $x=0$;
\item $\left\Vert \lambda x\right\Vert _{\delta }=\left\vert \lambda \right\vert ^{\delta }\left\Vert x\right\Vert _{\delta }$ for all $\lambda \in K$ and all $x\in E$;
\item $\left\Vert x+y\right\Vert _{\delta }\leq \left\Vert x\right\Vert _{\delta}+\left\Vert y\right\Vert _{\delta }$.
\end{enumerate}
\end{definition}

Let $J=\left[ 0,T\right] ,$ $J^{\prime }=\left( 0,T\right] $ and $C\left( J,\mathbb{R}\right)$ the space of continuous functions. The weighted space $C_{1-\gamma ;\psi }\left( J,\mathbb{R} \right) $ of functions $x$ on $J^{\prime }$ is defined by \cite{ZE1}
\begin{equation*}
C_{1-\gamma ;\psi }\left( J,\mathbb{R}\right) =\left\{ x\in C\left(J^{\prime },\mathbb{R}\right) ;\text{ }\left( \psi \left( t\right) -\psi \left( 0\right) \right) ^{1-\gamma }x\left( t\right) \in C\left( J,\mathbb{R}\right) \right\} ,\text{ }0\leq \gamma <1 
\end{equation*}
with the norm 
\begin{equation*}
\left\Vert x\right\Vert _{C_{1-\gamma ;\psi },\delta }=\underset{t\in J^{\prime }}{\sup }\left\{ \left( \psi \left( t\right) -\psi \left( 0\right) \right) ^{1-\gamma }\left\Vert x\left( t\right) \right\Vert _{\delta }\right\} .
\end{equation*}

Obviously the space $C_{1-\gamma ;\psi }\left( J,\mathbb{R} \right) $ is a Banach space.

The piecewise weighted space $PC_{1-\gamma ;\psi }\left( J,\mathbb{R}\right) $ of functions $x$ on $C\left( \left( t_{k},t_{k+1}\right] ,\mathbb{R} \right) $ is defined by 
\begin{equation*}
PC_{1-\gamma ;\psi }\left( J,\mathbb{R}\right) =\left\{ 
\begin{array}{l}
\left( \psi \left( t\right) -\psi \left( t_{k}\right) -\psi \left( 0\right)
\right) ^{1-\gamma }x\left( t\right) \in C\left( \left( t_{k},t_{k+1}\right]
,\mathbb{R}\right) \text{and }  \\ 
\underset{t\rightarrow t_{k}}{\lim }\left( \psi \left( t\right)
-\psi \left( t_{k}\right) -\psi \left( 0\right) \right) ^{1-\gamma }x\left(
t\right) ,\text{exists for }k=1,2,...,m%
\end{array}%
\right\} 
\end{equation*}
with norm 
\begin{equation*}
\qquad \left\Vert x\right\Vert _{PC_{1-\gamma ;\psi },\delta }:=\underset{ k=1,2,...,m}{\max }\left\{ \underset{t\in \left( t_{k},t_{k+1}\right] }{\sup  }\left( \psi \left( t\right) -\psi \left( t_{k}\right) -\psi \left( 0\right)
\right) ^{1-\gamma }\left\Vert x\left( t\right) \right\Vert _{\delta}\right\}
\end{equation*}
and there exists $x\left( t_{k}^{-}\right) $ and $x\left( t_{k}^{+}\right)$, $k=1,2,...,m$ with $x\left( t_{k}^{-}\right) =x\left( t_{k}^{+}\right)$. The space $PC_{1-\gamma ;\psi }\left( J,\mathbb{R} \right) $ is also a Banach space.

Let $n-1<\alpha \leq n$ with $n\in \mathbb{N}$, $J=\left[ a,b\right] $ be an interval such that $-\infty \leq a<b\leq
+\infty $ and let $f,\psi \in C^{n}\left( \left[ a,b\right],\mathbb{R} \right) $ be two functions such that $\psi $ is increasing and $\psi^{\prime }\left( t\right) \neq 0$, for all $t\in J.$ The $\psi -$Hilfer
fractional derivative denoted by \cite{ZE1}
\begin{equation*}
^{H}\mathbb{D}_{0+}^{\alpha ,\beta ;\psi }y\left( t\right) =I_{0+}^{\beta \left( n-\alpha \right) ;\psi }\left( \frac{1}{\psi ^{\prime }\left( t\right) }\frac{d}{dt}\right) ^{n}I_{0+}^{\left( 1-\beta \right) \left( n-\alpha \right) ;\psi }y\left( t\right) ,
\end{equation*}
where $I_{0+}^{\xi;\psi}(\cdot)$ ($0<\xi\leq 1$) is the $\psi$-Riemann-Liouville fractional integral \cite{ZE1}.

Banach's fixed-point theorem is fundamental in this study, so we will enunciate it below.

\begin{teorema} {\rm \cite{wangJin,ZE3}} Let $\left( X,d\right) $ be a generalized complete metric space. Assume that $\Omega :X\rightarrow X$ is a strictly contractive operator with the Lipschitz constant $L<1.$ If there exists a nonnegative integer $k$ such that $d\left( \Omega ^{k+1},\Omega ^{k}\right) <\infty $ for some \ $x\in X\, $, then the following are true:
\begin{enumerate}
\item The sequence $\left\{ \Omega ^{k}x\right\} $ converges to a fixed point $x^{\ast }$ of $\Omega$;

\item $x^{\ast }$ is the unique fixed point of $\Omega $ in $\Omega ^{\ast }=\left\{ y\in X/d\left( \Omega ^{k}x,y\right) <\infty \right\}$;

\item If $y\in X^{\ast }$, then 
\begin{equation*}
d\left( y,x^{\ast }\right) \leq \frac{1}{1-L}d\left( \Omega y,y\right) .
\end{equation*}
\end{enumerate}
\end{teorema}

Let the space of piecewise weighted space continuous functions 
\begin{equation*}
X=\left\{ g:J\rightarrow \mathbb{R} /g\in PC_{1-\gamma ;\psi }\left( J, \mathbb{R}\right) ,\text{ }0\leq \gamma <1\right\},
\end{equation*}
with the generalized metric on $X$ given by 
\begin{equation}\label{eq22}
d\left( g,h\right) =\inf \left\{ C_{1}+C_{2}\in \left[ 0,\infty \right]/\left\vert g\left( t\right) -h\left( t\right) \right\vert ^{\delta }\leq\left( C_{1}+C_{2}\right) \left( \varphi ^{\delta }\left( t\right)
+\varepsilon ^{\delta }\right),\text{ for all }t\in J\right\}
\end{equation}
where $C_{1}\in \left\{ C\in \left[ 0,\infty \right] /\left\vert g\left(t\right) -h\left( t\right) \right\vert ^{\delta }\leq C\varphi ^{\delta }\left( t\right) ,\text{ for all }t\in \left( s_{i},t_{i+1}\right],
\text{ }i=0,1,...,m\right\} $ and $C_{2}\in \left\{ C\in \left[ 0,\infty \right] /\left\vert g\left( t\right) -h\left( t\right) \right\vert ^{\delta }\leq C\varepsilon ^{\delta },\text{ for all }t\in \left( t_{i},s_{i}\right], \text{ }i=1,...,m\right\}$. Note that that $\left( X,d\right)$ is a complete generalized metric space.

The function $x\in \widetilde{U}:=PC_{1-\gamma ;\psi }\left( J,\mathbb{R}\right)\overset{m}{\underset{i=0}{\displaystyle\bigcap }}C^{1}\left( \left( s_{i},t_{i+1}\right], \mathbb{R}\right)$ is a solution of the impulsive fractional differential equations
\begin{equation*}
\left\{ 
\begin{array}{cll}
^{H}\mathbb{D}_{0+}^{\alpha ,\beta ;\psi }x\left( t\right) & = & f\left( t,x\left( t\right) \right) ,\text{ }t\in \left( s_{i},t_{i+1}\right] ,\text{ }i=0,1,...,m \\ 
x\left( t\right) & = & g_{i}\left( t,x\left( t_{i}^{+}\right) \right) ,\text{ }t\in \left( t_{i},s_{i}\right] \,,\text{ }i=1,2,...,m \\  I_{0+}^{1-\gamma ;\psi }x\left( 0\right) & = & x_{0}\in \mathbb{R}
\end{array}%
\right. 
\end{equation*}
where $I^{1-\gamma;\psi}_{0+}(\cdot)$ is the $\psi$-Riemann-Liouville fractional integral with $\gamma=\alpha+\beta(1-\alpha)$, if $x$ satisfies $I_{0+}^{1-\gamma ;\psi }x\left( 0\right) =x_{0}$, $x\left( t\right) =g_{i}\left( t,x\left( t_{i}^{+}\right) \right) ,$ $t\in \left( t_{i},s_{i}\right] \,,$ $i=1,2,...,m$ and 
\begin{equation*}
\left\{ 
\begin{array}{cll}
x\left( t\right) & = & \Psi^{\lambda}(t,0) x_{0}+\dfrac{1}{ \Gamma \left( \alpha \right) }\displaystyle\int_{0}^{t}N_{\psi }^{\alpha }\left( t,s\right) f\left( s,x\left( s\right) \right) ds,\text{ }t\in \left[ 0,t_{1}\right] \\ 
x\left( t\right) & = & g_{i}\left( s_{i},x\left( t_{i}^{+}\right) \right) +\dfrac{1}{\Gamma \left( \alpha \right) }\displaystyle\int_{s_{i}}^{t}N_{\psi }^{\alpha }\left( t,s\right) f\left( s,x\left( s\right) \right) ds,\text{ }t\in \left( s_{i},t_{i+1}\right] ,\text{ }i=1,2,...,m.
\end{array}%
\right.
\end{equation*}
with $N_{\psi }^{\alpha }\left( s,t\right) :=\psi ^{\prime }\left( s\right) \left( \psi \left( t\right) -\psi \left( s\right) \right) ^{\alpha -1}$ and $\Psi^{\lambda}(t,0)= \dfrac{(\psi(t)-\psi(0))^{\gamma-1}}{\Gamma(\gamma)}$.

Let $0<\delta \leq 1,$ $\xi \geq 0$, $\varphi \in PC_{1-\gamma ;\psi }\left( J,\mathbb{R}_{+}\right) $ is nondecreasing and
\begin{equation}\label{eq2}
\left\{ 
\begin{array}{cll}
\left\vert ^{H}\mathbb{D}_{0+}^{\alpha ,\beta ;\psi }y\left( t\right) -f\left( t,y\left( t\right) \right) \right\vert & \leq & \varphi \left( t\right), \text{ }t\in \left( s_{i},t_{i+1}\right] ,\text{ }i=0,1,...,m \\ 
\left\vert y\left( t\right) -g_{i}\left( t,y\left( t_{i}^{+}\right) \right) \right\vert & \leq & \xi ,\text{ }t\left( t_{i},s_{i}\right] ,i=1,2,....,m
\end{array}
\right.
\end{equation}

\begin{definition} The {\rm Eq.(\ref{eq1})} is generalized $\delta -$Ulam-Hyers-Rassias stable with respect to $\left( \varphi ,\xi \right) $ if there exists $C_{f,\delta ,g_{i},\varphi }>0$ such that for each solution $y\in \widetilde{U}$ of the inequality {\rm Eq.(\ref{eq2})} there exists a solution $x\in \widetilde{U}$ of the
{\rm Eq.(\ref{eq1})} with 
\begin{equation*}
\left\vert y\left( t\right) -x\left( t\right) \right\vert ^{\delta }\leq C_{f,\delta ,g_{i},\varphi }\left( \xi ^{\delta }+\varphi ^{\delta }\left( t\right) \right) ,\text{ }t\in J.
\end{equation*}
\end{definition}

A function $y\in \widetilde{U}$ is a solution of the inequality {\rm Eq.(\ref{eq2})} if and only if there is $G\in \underset{i=0}{\overset{m}{\displaystyle\bigcap }} C^{1}\left( \left( s_{i},t_{i+1}\right],\mathbb{R} \right) $ and $g\in \underset{i=0}{\overset{m}{\displaystyle\bigcap }}C\left( \left[t_{i},s_{i}\right], \mathbb{R} \right) $, such that:

{\rm (a)} $\left\vert G\left( t\right) \right\vert \leq \varphi \left( t\right)$, $t\in \overset{m}{\underset{i=0}{\displaystyle\bigcup }}\left( s_{i},t_{i+1}\right] $ and  $\left\vert g\left( t\right) \right\vert \leq \xi ,$ $t\in \overset{m}{\underset{i=0}{\displaystyle\bigcup }}\left( t_{i},s_{i}\right]$;

{\rm (b)} $^{H}\mathbb{D}_{0+}^{\alpha ,\beta ;\psi }y\left( t\right) =f\left( t,y\left( t\right) \right) +G\left( t\right) ,$ $t\in \left( s_{i},t_{i+1}\right]$, $i=0,1,...,m$;

{\rm (c)} $y\left( t\right) =g_{i}\left( t,y\left( t_{i}^{+}\right) \right)+g\left( t\right)$, $t\in \left( t_{i},s_{i}\right] ,$ $i=1,2,...,m$.

\begin{remark} If $y\in \widetilde{U}$ is a solution of the inequality {\rm Eq.(\ref{eq2})} the $y$ is a solution of the fractional integral inequality: 
\begin{equation}\label{eq14}
\left\vert y\left( t\right) -g_{i}\left( t,y\left( t_{i}^{+}\right) \right) \right\vert  \leq  \xi ,\text{ }t\in \left( t_{i},s_{i}\right] ,\text{ }i=0,1,...,m
\end{equation}
and
\begin{eqnarray}\label{eq141}
&&\left\vert y\left( t\right) -\Psi ^{\lambda }(t,0)y\left( 0\right) -\dfrac{1}{\Gamma \left( \alpha \right) }\displaystyle\int_{0}^{t}N_{\psi }^{\alpha }\left( t,s\right) f\left( s,y\left( s\right) \right) ds\right\vert \notag \leq \dfrac{1}{\Gamma \left( \alpha \right) }\displaystyle\int_{0}^{t}N_{\psi }^{\alpha }\left( t,s\right) \varphi \left( s\right) ds,\text{ }t\in \left[ 0,t_{1}\right] 
\end{eqnarray}
and 
\begin{eqnarray}\label{eq142}
&&\left\vert y\left( t\right) -g_{i}\left( s_{i},y\left( t_{i}^{+}\right) \right) -\dfrac{1}{\Gamma \left( \alpha \right) }\displaystyle\int_{s_{i}}^{t}N_{\psi }^{\alpha }\left( t,s\right) f\left( s,y\left(
s\right) \right) ds\right\vert \notag \\
&\leq &\xi +\dfrac{1}{\Gamma \left( \alpha \right) }\displaystyle\int_{s_{i}}^{t}N_{\psi }^{\alpha }\left( t,s\right) \varphi \left( s\right)ds,\text{ }t\in \left( s_{i},t_{i+1}\right] ,\text{ } i=1,2,...,m.
\end{eqnarray}

\end{remark}


\section{$\delta -$Ulam-Hyers-Rassias stability}
In this section, we present the main result of this paper, the stability of the type generalized $\delta-$Ulam-Hyers-Rassias for the impulsive fractional differential equation Eq.(\ref{eq1}), by means Banach's the fixed point theorem.

Before investigating the main result in this paper, we list some essential conditions for proof of the theorem:

{\rm (H1)} $f\in C_{1-\gamma ;\psi }\left( J\times \mathbb{R},\mathbb{R}\right)$;

{\rm (H2)} There exists a positive constant $L_{f}$ such that 
\begin{equation*}
\left\vert f\left( t,u_{1}\right) -f\left( t,u_{2}\right) \right\vert \leq L_{f}\left\vert u_{1}-u_{2}\right\vert, \text{ } \mbox{for $t\in J$ and $u_{1},u_{2}\in \mathbb{R}$;}
\end{equation*}

{\rm (H3)} $g_{i}\in C_{1-\gamma ;\psi }\left( \left[ t_{i},s_{i}\right] \times \mathbb{R},\mathbb{R}\right) $ and there are positive constants $L_{g_{i}}$,  $i=1,2,...,m$ such that 
\begin{equation*}
\left\vert g_{i}\left( t,u_{1}\right) -g_{2}\left( t,u_{2}\right) \right\vert \leq L_{g_{i}}\left\vert u_{1}-u_{2}\right\vert, \text{ } \mbox{for $\ t\in \left[ t_{i},s_{i}\right] $ and $u_{1},u_{2}\in \mathbb{R}$;}
\end{equation*}

{\rm (H4)} Let $\varphi \in C_{1-\gamma ;\psi }\left( J,\mathbb{R}_{+}\right) $ be a nondecreasing function. There exists $C_{\varphi }>0$ such that 
\begin{equation*}
\frac{1}{\Gamma \left( \alpha \right) }\int_{0}^{t}N_{\psi }^{\alpha }\left(t,s\right) \varphi \left( s\right) ds\leq c_{\varphi }\varphi \left( t\right), \text{ } \mbox{ for $t\in J$.}
\end{equation*}

\begin{teorema}\label{teo2} Assume the conditions {\rm (H1)-(H4)} be satisfies. If there exists a function $y\in \widetilde{U}$  satisfying {\rm Eq.(\ref{eq2})}, then there exists a unique solution $y_{0}:J\rightarrow \mathbb{R}$ such that
\begin{equation}\label{eq3}
y_{0}\left( t\right) =\left\{ 
\begin{array}{l}
\Psi^{\lambda}(t,0)x\left( 0\right) +\displaystyle\frac{1}{\Gamma \left( \alpha \right) }\displaystyle\int_{0}^{t}N_{\psi }^{\alpha }\left(t,s\right) f\left( s,y_{0}\left( s\right) \right) ds,\text{ }t\in \left[ 0,t_{1}\right]  \\ 
g_{i}\left( t,y_{0}\left( t_{i}^{+}\right) \right) ,\text{ }t\in \left( t_{i},s_{i}\right] ,\text{ }i=1,2,...,m \\  g_{i}\left( s_{i},y_{0}\left( t_{i}^{+}\right) \right) +\dfrac{1}{\Gamma \left( \alpha \right) }\displaystyle\int_{s_{i}}^{t}N_{\psi }^{\alpha }\left(t,s\right) f\left( s,y_{0}\left( s\right) \right) ds,\text{ }t\in \left( s_{i},t_{i+1}\right]  \\ 
i=1,2,...,m.
\end{array}
\right. 
\end{equation}
and 
\begin{equation}\label{eq56}
\left\vert y\left( t\right) -y_{0}\left( t\right) \right\vert ^{\delta }\leq  \frac{\left( 1+C_{\varphi }^{\delta }\right) \left( \varphi ^{\delta }\left( t\right) +\xi ^{\delta }\right) }{1-\Phi },\text{ }t\in J
\end{equation}
where 
\begin{equation}\label{eq25}
\Phi :=\underset{i=1,2,...,m}{\max }\left\{ L_{g_{i}}^{\delta }+L_{f}^{\delta }C_{\varphi }^{\delta }\right\} .
\end{equation}
\end{teorema}

\begin{proof}
For prove of this result, consider the operator $ \Omega :X\rightarrow X$ given by
\begin{equation}\label{eq4}
\Omega x\left( t\right) =\left\{ 
\begin{array}{l}
\Psi^{\lambda}(t,0)x\left( 0\right) +\displaystyle\frac{1}{\Gamma \left( \alpha \right) }\displaystyle\int_{0}^{t}N_{\psi }^{\alpha }\left(t,s\right) f\left( s,x_{0}\left( s\right) \right) ds,\text{ }t\in \left[ 0,t_{1}\right]  \\ 
g_{i}\left( t,x_{0}\left( t_{i}^{+}\right) \right) ,\text{ }t\in \left( t_{i},s_{i}\right], \text{ }i=1,2,...,m \\  g_{i}\left( s_{i},x_{0}\left( t_{i}^{+}\right) \right) +\dfrac{1}{\Gamma \left( \alpha \right) }\displaystyle\int_{s_{i}}^{t}N_{\psi }^{\alpha }\left(t,s\right) f\left( s,x_{0}\left( s\right) \right) ds, \\ 
\text{ }t\in \left( s_{i},t_{i+1}\right],\text{ } i=1,2,...,m.
\end{array}
\right. 
\end{equation}
for all $x\in X$ and $t\in \left[ 0,T\right]$. Note that, $\Omega $ is a well defined operator according with condition (H1).

In order to use Banach's fixed point theorem, we prove that $\Omega $ is strictly contractive on $X,$ in three cases. Note that 
\begin{equation*}
\left\vert g\left( t\right) -h\left( t\right) \right\vert ^{\delta }\leq
\left\{ 
\begin{array}{l}
C_{1}\varphi ^{\delta }\left( t\right) ,\text{ }t\in \left( s_{i},t_{i+1} \right] ,\text{ }i=0,1,...,m \\ 
C_{2}\xi ^{\delta },\text{ }t\in \left( t_{i},s_{i}\right] ,\text{ }i=1,...,m
\end{array}%
\right.
\end{equation*}
is equivalent to 
\begin{equation}\label{eq5}
\left\vert g\left( t\right) -h\left( t\right) \right\vert \leq \left\{ 
\begin{array}{l}
C_{1}^{\delta }\varphi \left( t\right) ,\text{ }t\in \left( s_{i},t_{i+1}\right] ,\text{ }i=0,1,...,m \\ 
C_{2}^{\delta }\xi ,\text{ }t\in \left( t_{i},s_{i}\right] ,\text{ }i=1,...,m
\end{array}%
\right.
\end{equation}

Using the definition of $\Omega $ in Eq.(\ref{eq4}), (H2), (H3) and Eq.(\ref{eq5}), we have the following cases:

Case 1: For $t\in \left[ 0,t_{1}\right] $, and by the hypothesis (H2), (H4) and Eq.(\ref{eq5}), we get 
\begin{eqnarray*}
\left\vert \Omega g\left( t\right) -\Omega h\left( t\right) \right\vert ^{\delta }  \notag &\leq &\left\vert \frac{1}{\Gamma \left( \alpha \right) }\int_{0}^{t}N_{\psi }^{\alpha }\left(t,s\right) f\left( s,g\left( s\right) \right) ds-\frac{1}{\Gamma \left( \alpha \right) }\int_{0}^{t}N_{\psi }^{\alpha }\left(t,s\right) f\left( s,h\left( s\right) \right) ds\right\vert ^{\delta } \notag \\
&\leq &L_{f}^{\delta }\left( \frac{1}{\Gamma \left( \alpha \right) }\int_{0}^{t}N_{\psi }^{\alpha }\left(t,s\right) \left\vert g\left( s\right) -h\left( s\right) \right\vert ds\right) ^{\delta } \leq L_{f}^{\delta }C_{1}c_{\varphi }^{\delta }\varphi \left( t\right)
^{\delta }.
\end{eqnarray*}

Case 2: For $t\in \left( t_{i},s_{i}\right] ,$ and by the hypothesis (H3) and Eq.(\ref{eq4}), we obtain 
\begin{eqnarray*}
\left\vert \Omega g\left( t\right) -\Omega h\left( t\right) \right\vert ^{\delta }  &\leq &\left\vert g_{i}\left( t,g\left( t_{i}^{+}\right) \right) -g_{i}\left( t,h\left( t_{i}^{+}\right) \right) \right\vert ^{\delta } \leq \left( L_{g_{i}}\left\vert g\left( t_{i}^{+}\right) -h\left( t_{i}^{+}\right) \right\vert \right) ^{\delta } \leq L_{g_{i}}^{\delta }C_{2}\xi ^{\delta }.
\end{eqnarray*}

Case 3: For $t\in \left( s_{i},t_{i+1}\right] $ and using the hypothesis (H1), (H2), (H3) and Eq.(\ref{eq4}), we have 
\begin{eqnarray*}
&&\left\vert \Omega g\left( t\right) -\Omega h\left( t\right) \right\vert ^{\delta }  \notag \\
&\leq &\left\vert g_{i}\left( s_{i},g\left( t_{i}^{+}\right) \right) -g_{i}\left( s_{i},h\left( t_{i}^{+}\right) \right) \right\vert ^{\delta }+\left\vert \frac{1}{\Gamma \left( \alpha \right) }\int_{s_{i}}^{t}N_{\psi
}^{\alpha }\left(t,s\right) \left( f\left( s,g\left( s\right) \right) -f\left( s,h\left( s\right) \right) \right) ds\right\vert ^{\delta }  \notag \\
&\leq &L_{g_{i}}^{\delta }C_{2}\xi ^{\delta }+L_{f}^{\delta }C_{1}\left(  \frac{1}{\Gamma \left( \alpha \right) }\int_{0}^{t}N_{\psi }^{\alpha }\left(t,s\right) \varphi \left( s\right) ds\right) ^{\delta }  
\leq \left( L_{g_{i}}^{\delta }+L_{f}^{\delta }c_{\varphi }^{\delta }\right) \left( C_{1}+C_{2}\right) \left( \varphi ^{\delta }\left( t\right) +\xi ^{\delta }\right) .
\end{eqnarray*}

Then, we obtain
\begin{eqnarray*}
\left\vert \Omega g\left( t\right) -\Omega h\left( t\right) \right\vert ^{\delta } &\leq &\underset{i=1,2,...,m}{\max }\left( L_{g_{i}}^{\delta }+L_{f}^{\delta }C_{\varphi }^{\delta }\right) \left( C_{1}+C_{2}\right)
\left( \varphi ^{\delta }\left( t\right) +\xi ^{\delta }\right)  \notag \\ &=&\Phi \left( C_{1}+C_{2}\right) \left( \varphi ^{\delta }\left( t\right) +\xi ^{\delta }\right) ,\text{ }t\in J.
\end{eqnarray*}

Hence, we get 
\begin{equation*}
d\left( \Omega g,\Omega h\right) \leq \Phi d\left( g,h\right)
\end{equation*}
for any $g,h\in X$ and since the condition Eq.(\ref{eq25}).

Now, we take $g_{0}\in X$ and from the piecewise continuous property of $g_{0}$ and $\Omega g_{0}$, then there exists a constant $ 0<G_{1}<\infty $ so that 
\begin{eqnarray*}
\left\vert \Omega g_{0}\left( t\right) -g_{0}\left( t\right) \right\vert ^{\delta } &=&\left\vert \Psi^{\lambda}(t,0)x\left( 0\right) +\frac{1 }{\Gamma \left( \alpha \right) }\int_{0}^{t}N_{\psi }^{\alpha }\left(t,s\right) f\left( s,g_{0}\left( s\right) \right) ds-g_{0}\left( t\right) \right\vert ^{\delta }  \notag \\&\leq &G_{1}\varphi ^{\delta }\left( t\right) \leq G_{1}\left( \varphi ^{\delta }\left( t\right) +\xi ^{\delta }\right) ,\text{ }t\in \left[ 0,t_{1}\right] .
\end{eqnarray*}

On the other hand, also $G_{2}$ and $G_{3}$ with $0<G_{2}<\infty $ and $0<G_{3}<\infty ,$ such that, 
\begin{eqnarray*}
\left\vert \Omega g_{0}\left( t\right) -g_{0}\left( t\right) \right\vert ^{\delta } &=&\left\vert g_{i}\left( t,g_{0}\left( t_{i}^{+}\right) \right) -g_{0}\left( t\right) \right\vert ^{\delta }  \notag \\
&\leq &G_{2}\xi ^{\delta }\leq G_{2}\left( \varphi ^{\delta }\left( t\right)+\xi ^{\delta }\right) ,\text{ }t\in \left( t_{i},s_{i}\right] ,\text{ }i=1,2,...,m
\end{eqnarray*}
and
\begin{eqnarray*}
\left\vert \Omega g_{0}\left( t\right) -g_{0}\left( t\right) \right\vert ^{\delta } &=&\left\vert g_{i}\left( s_{i},g_{0}\left( t_{i}^{+}\right) \right) +\frac{1}{\Gamma \left( \alpha \right) }\int_{s_{i}}^{t}N_{\psi
}^{\alpha }\left(t,s\right) f\left( s,g_{0}\left( s\right) \right) ds-g_{0}\left( t\right) \right\vert ^{\delta }  \notag \\ &\leq &G_{3}\left( \varphi ^{\delta }\left( t\right) +\xi ^{\delta }\right), \text{ }t\in \left( s_{i},t_{i+1}\right] ,\text{ }i=1,2,...,m
\end{eqnarray*}
since $f,g_{i}$ and $g_{0}$ are bounded on $J$ and $\varphi \left( \cdot \right) +\xi ^{\delta }>0$. In this sense, Eq.(\ref{eq22}) implies that
\begin{equation*}
d\left( \Omega g_{0},g_{0}\right) <\infty .
\end{equation*}

Note that, exists a continuous function $y_{0}:J\rightarrow \mathbb{R} $ such that $\Omega ^{n}y_{0}\rightarrow y_{0}$ in $\left( X,d\right) $ as $n\rightarrow \infty $ and $\Omega y_{0}=y_{0}$, that is $y_{0}$ satisfies Eq.(\ref{eq4}) for every $t\in J$ (Banach's fixed point theorem).

For finally the proof this theorem, we check that $0<C_{g}<\infty $ such that 
\begin{equation*}
\left\vert g_{0}\left( t\right) -g\left( t\right) \right\vert ^{\delta }\leq C_{g}\left( \varphi ^{\delta }\left( t\right) +\xi ^{\delta }\right), \text{ } \mbox{for any $t\in J$}
\end{equation*}
and assuming that $g$, $g_{0}$ are bounded on $J$ and $\underset{t\in J}{\min }\left( \varphi ^{\delta }\left( t\right) +\xi ^{\delta }\right) >0$.

Then, we get $d\left( g_{0},g\right) <\infty $ for all $g\in X$, that is $X=\left\{ g\in X/d\left( g_{0},g\right) <\infty \right\}.$ Therefore, we obtain that $y_{0}$ is the unique solution continuous function with the property Eq.(\ref{eq4}).

On the other hand, using the hypotheses (H1)-(H4), Eq.(\ref{eq14}), Eq.(\ref{eq141}) and Eq.(\ref{eq142}) it follows that 
\begin{equation}\label{eq13}
d\left( y,\Omega y\right) \leq 1+C_{\varphi }^{\delta }.
\end{equation}

Thus, from Eq.(\ref{eq13}), we have
\begin{equation*}
d\left( y,y_{0}\right) \leq \frac{d\left( \Omega y,y\right) }{1-\Phi }\leq \frac{1+C_{\varphi }^{\delta }}{1-\Phi },
\end{equation*}
which means that Eq.(\ref{eq56}) is true for $t\in J$ .
\end{proof}

\bibliography{ref}

\begin{thebibliography}{10}

\bibitem{theory}
V.~Lakshmikantham, P.~S. Simeonov, Theory of impulsive differential equations,
  Vol.~6, World scientific, 1989.

\bibitem{WangZhang}
J.~Wang, Y.~Zhang, Existence and stability of solutions to nonlinear impulsive
  differential equations in $\beta$-normed spaces, Electronic J. Diff.
  Equations 2014~(83) (2014) 1--10.

\bibitem{wangJin}
J.~Wang, Z.~Lin, Y.~Zhou, On the stability of new impulsive ordinary
  differential equations, Topol. Methods Nonlinear Anal. 46~(1) (2015)
  303--314.

\bibitem{princi}
Z.~Lin, W.~Wei, J.~Wang, Existence and stability results for impulsive
  integro-differential equations, Facta Universitatis 29~(2) (2014) 119--130.

\bibitem{SAMKO}
S.~G. Samko, A.~A. Kilbas, O.~I. Marichev, Fractional integrals and
  derivatives, Theory and Applications, Gordon and Breach, Yverdon 1993 (1993)
  44.

\bibitem{KSTJ}
A.~A. Kilbas, H.~M. Srivastava, J.~J. Trujillo, Theory and {A}pplications of
  {F}ractional {D}ifferential {E}quations, Vol. 204, Elsevier, Amsterdam, 2006.

\bibitem{ZE1}
J.~Vanterler~da C.~Sousa, E.~Capelas~de Oliveira, On the $\psi$-{H}ilfer
  fractional derivative, Commun. Nonlinear Sci. Numer. Simulat. 60 (2018)
  72--91.

\bibitem{dicheng}
Q.~Wang, D.~Lu, Y.~Fang, Stability analysis of impulsive fractional
  differential systems with delay, Appl. Math. Lett. 40 (2015) 1--6.

\bibitem{stamova}
I.~Stamova, Global stability of impulsive fractional differential equations,
  Appl. Math. and Comput. 237 (2014) 605--612.

\bibitem{liu}
K.~Liu, W.~Jiang, Stability of nonlinear {C}aputo fractional differential
  equations, Appl. Math. Modelling 40~(5-6) (2016) 3919--3924.

\bibitem{esto9}
J.~Wang, W.~Wei, Y.~Yang, On some impulsive fractional differential equations
  in {B}anach spaces, Opuscula Mathematica 30~(4) (2010) 507--525.

\bibitem{esto10}
T.~L. Guo, Nonlinear impulsive fractional differential equations in {B}anach
  spaces, Topol. Methods Nonlinear Anal. 42~(1) (2013) 221--232.

\bibitem{palm}
J.~Vanterler~da C.~Sousa, D.~S. Oliveira, E.~Capelas~de Oliveira, On the
  existence and stability for impulsive fractional integrodifferential
  equation, arXiv:1806.01442, (2018).

\bibitem{esto3}
A.~Sivasankari, A.~Leelamani, Existence of mild solutions for an impulsive
  fractional neutral integro-differential equations with non-local conditions
  in {B}anach spaces., Nonlinear Studies 24~(3).

\bibitem{esto5}
P.~Kumar, R.~Haloi, D.~Bahuguna, D.~N. Pandey, Existence of solutions to a new
  class of abstract non-instantaneous impulsive fractional integro-differential
  equations, Nonl. Dyn. and Systems Theory (2016) 73.

\bibitem{esto11}
M.~Benchohra, D.~Seba, Impulsive fractional differential equations in {B}anach
  spaces, Electron. J. Qual. Theory Differ. Equ 8~(1).

\bibitem{exis7}
H.~Gou, B.~Li, Local and global existence of mild solution to impulsive
  fractional semilinear integro-differential equation with noncompact
  semigroup, Commun. Nonlinear Sci. Numer. Simulat. 42 (2017) 204--214.

\bibitem{Wang01}
J.~Wang, M.~M. Feckan, Y.~Zhou, A survey on impulsive fractional differential
  equations, Fract. Calc. Appl. Anal. 16~(4) (2016) 806--831.

\bibitem{WangZh}
J.~Wang, Y.~Zhang, A class of nonlinear differential equations with fractional
  integrable impulses, Commun. Nonlinear Sci. Numer. Simulat. 19~(2) (2014)
  3001--3010.

\bibitem{ZE2}
J.~Vanterler~da C.~Sousa, E.~Capelas~de Oliveira, A {G}ronwall inequality and
  the {C}auchy-type problem by means of $\psi$-{H}ilfer operator,
  arXiv:1709.03634.

\bibitem{ZE3}
J.~Vanterler~da C.~Sousa, E.~Capelas~de Oliveira, Ulam–-{H}yers stability of
  a nonlinear fractional {V}olterra integro-differential equation, Appl. Math.
  Lett. 81 (2018) 50--56.

\end{thebibliography}
\bibliographystyle{plain}

\end{document}